\RequirePackage{ifpdf} 
\ifpdf
\documentclass[10pt,pdftex]{amsart}
\else
\documentclass[10pt,dvips]{amsart}
\fi

\ifpdf
\usepackage{lmodern}
\fi
\usepackage[bookmarks, 
colorlinks=false, backref=false,plainpages=false]{hyperref}
\usepackage[T1]{fontenc}
\usepackage[latin1]{inputenc}
\usepackage[english]{babel}
\usepackage{graphicx}
\usepackage{url}
\usepackage{graphics}
\usepackage{epsfig}
\usepackage{amsmath,amssymb,amsthm}

\usepackage{enumerate}

\newtheorem{thm}{Theorem}[section]

\newtheorem{lem}[thm]{Lemma}

\newtheorem{rem}[thm]{Remark}
\newtheorem{assumption}[thm]{Assumption}

\usepackage{xcolor}

\begin{document}
\newcommand{\BX}{{\bf X}}
\newcommand{\cv}{{\cal V}}
\newcommand{\cW}{{\cal W}}
\newcommand{\co}{{\cal O}}

\def\abstract{
\advance \rightskip by 10mm
\advance \leftskip by 10mm
\vspace{-0.8em}
\noindent
\small{\bf Abstract.}
}
\def\endabstract{\par\normalsize\rm}

\def\Xint#1{\mathchoice
{\XXint\displaystyle\textstyle{#1}}%
{\XXint\textstyle\scriptstyle{#1}}%
{\XXint\scriptstyle\scriptscriptstyle{#1}}%
{\XXint\scriptscriptstyle\scriptscriptstyle{#1}}%
\!\int}
\def\XXint#1#2#3{{\setbox0=\hbox{$#1{#2#3}{\int}$}
\vcenter{\hbox{$#2#3$}}\kern-.5\wd0}}
\def\ddashint{\Xint=}
\def\dashint{\Xint-}

\def\a{\alpha}
\def\b{\beta}
\def\d{\delta}\def\D{\Delta}
\def\e{\epsilon}
\def\g{\gamma}\def\G{\Gamma}
\def\k{\kappa}
\def\lam{\lambda}\def\Lam{\Lambda}
\renewcommand\o{\omega}\renewcommand\O{\Omega}
\def\s{\sigma}\def\S{\Sigma}
\renewcommand\t{\theta}\def\vt{\vartheta}
\newcommand{\vphi}{\varphi}
\def\z{\zeta}

\newcommand{\tsigma}{\tilde{\s}}
\newcommand{\tbsigma}{\tilde{\bsigma}}
\def\te{\tilde{\e}}
\def\tu{\tilde{u}}

\newcommand{\bchi}{\mbox{\boldmath$\chi$}}
\newcommand{\bdelta}{\mbox{\boldmath$\delta$}}
\newcommand{\bepsilon}{\mbox{\boldmath$\epsilon$}}
\newcommand{\bfeta}{\mbox{\boldmath$\eta$}}
\newcommand{\bgamma}{\mbox{\boldmath$\gamma$}}
\newcommand{\bomega}{\mbox{\boldmath$\omega$}}
\newcommand{\bvphi}{\mbox{\boldmath$\varphi$}}
\newcommand{\bphi}{\mbox{\boldmath$\phi$}}
\newcommand{\bPhi}{\mbox{\boldmath$\Phi$}}
\newcommand{\bpsi}{\mbox{\boldmath$\psi$}}
\newcommand{\bPsi}{\mbox{\boldmath$\Psi$}}
\newcommand{\bsigma}{\mbox{\boldmath$\sigma$}}
\newcommand{\btau}{\mbox{\boldmath$\tau$}}
\newcommand{\bxi}{\mbox{\boldmath$\xi$}}
\newcommand{\brho}{\mbox{\boldmath$\rho$}}
\newcommand{\bbeta}{\mbox{\boldmath$\beta$}}
\newcommand{\bzeta}{\mbox{\boldmath$\zeta$}}

\def\bk{\boldsymbol{\kappa}}
\def\bmu{\boldsymbol\mu}
\def\bxi{\boldsymbol{\xi}}
\def\bz{\boldsymbol{\zeta}}

\def\ba{{\bf a}}
\def\bb{{\bf b}}
\def\bc{{\bf c}}
\def\be{{\bf e}}
\def\bff{{\bf f}}
\def\bg{{\bf g}}
\def\bn{{\bf n}}
\def\bp{{\bf p}}
\def\bq{{\bf q}}
\def\bs{{\bf s}}
\def\bt{{\bf t}}
\def\bu{{\bf u}}
\def\bv{{\bf v}}
\def\bw{{\bf w}}
\def\bx{{\bf x}}
\def\by{{\bf y}}
\def\bzz{{\bf z}}

\def\bD{{\bf D}}
\def\bE{{\bf E}}
\def\bF{{\bf F}}
\def\bH{{\bf H}}
\def\bJ{{\bf J}}
\def\bV{{\bf V}}
\def\bU{{\bf U}}
\def\bW{{\bf W}}
\def\bX{{\bf X}}
\def\bY{{\bf Y}}

\def\cA{{\cal A}}
\def\cC{{\cal C}}
\def\cD{{\cal D}}
\def\cE{{\cal E}}
\def\cF{{\cal F}}
\def\cG{{\cal G}}
\def\cI{{\cal I}}
\def\cJ{{\cal J}}
\def\cK{{\cal K}}
\def\cL{{\cal L}}
\def\cO{{\cal O}}
\def\cP{{\cal P}}
\def\cQ{{\cal Q}}
\def\cR{{\cal R}}
\def\cS{{\cal \Sigma}}
\def\cT{{\cal T}}
\def\cU{{\cal U}}
\def\cV{{\cal V}}

\def\scT{{_\cT}}
\def\sD{{_D}}
\def\sE{{_E}}
\def\sF{{_F}}
\def\sFz{{_{F_z}}}
\def\sK{{_K}}
\def\sI{{_I}}
\def\sb{{_b}}
\def\sN{{_N}}

\def\curl{{{\bf curl} \ }}
\def\rot{{\mbox{rot}\ }}
\def\BPI{{\bf \Pi}}

\def\cth{\cT_h}
\def\ctH{\cT_H}

\def\tJ{\tilde{\J}}

\def\hK{\widehat{K}}
\def\hx{\widehat{x}}
\def\hy{\widehat{y}}
\def\bhv{\widehat{\bv}}

\def\l{\ell}
\def\bl{\boldsymbol{\ell}}
\def\col{\colon}
\def\f12{\frac12}
\def\dfrac{\displaystyle\frac}
\def\dint{\displaystyle\int}
\def\nab{\nabla}
\def\p{\partial}
\def\sm{\setminus}
\def\dsum{\displaystyle\sum}
\newcommand{\pp}[2]{\frac{\partial {#1}}{\partial {#2}}}
\def\bzero{{\bf 0}}

\def\divv{\nab\cdot}
\def\divx{\nab_x\cdot}
\def\divtx{\nab_{t,x}\cdot}
\def\nabx{\nab_x}

\newcommand{\grad}{\nabla}
\newcommand{\curlt}{{\nabla \times}}
\newcommand{\gperp}{\nabla^{\perp}}
\newcommand{\gradt}{\nabla\cdot}

\def\forallqq{\quad\forall\,}
\def\aph{A^{1/2}}
\def\amh{A^{-1/2}}

\def\osc{{\rm osc \, }}

\def\Im{{\rm Im}}
\newcommand{\tr}{{\rm tr}}
\def\divvr{{\rm div}}
\def\curllr{{\rm curl}}
\def\curll{{\rm curl}}
\def\curl{{\bf curl}}
\newcommand{\bgrad}{{\bf grad}}
\newcommand\diam{\mathrm{diam\,}}
\renewcommand\Im{\mathrm{Im\,}}
\def\Span{\mbox{Span}}
\def\supp{\mbox{supp\,}}
\newcommand{\trace}{{\rm trace}}

\newcommand{\tri}{|\!|\!|}
\newcommand{\ljump}{\lbrack\!\lbrack}
\newcommand{\rjump}{\rbrack\!\rbrack}
\newcommand{\bdm}{\begin{displaymath}}
\newcommand{\edm}{\end{displaymath}}
\newcommand{\beq}{\begin{equation}}
\newcommand{\eeq}{\end{equation}}
\newcommand{\beqa}{\begin{eqnarray}}
\newcommand{\eeqa}{\end{eqnarray}}
\newcommand{\beqas}{\begin{eqnarray*}}
\newcommand{\eeqas}{\end{eqnarray*}}
\newcommand{\ul}{\underline}
\newcommand{\wh}{\widehat}
\newcommand{\la}{\langle}
\newcommand{\ra}{\rangle}

\newcommand{\Lt}{L^2(\Omega)}
\newcommand{\Lts}{L^2(\Omega)^2}
\newcommand{\Ltc}{L^2(\Omega)^3}
\newcommand{\Ho}{H^1(\Omega)}
\newcommand{\Hoh}{H^1(\wh{\Omega})}
\newcommand{\Hoi}{H^1(\Omega_i)}
\newcommand{\Hos}{H^1(\Omega)^2}
\newcommand{\Hoc}{H^1(\Omega)^3}
\newcommand{\Hoch}{H^1(\wh{\Omega})^3}
\newcommand{\Hoci}{H^1(\Omega_i)^3}
\newcommand{\Hoz}{H^1_0(\Omega)}
\newcommand{\Ht}{H^2(\Omega)}
\newcommand{\Hti}{H^2(\Omega_i)}
\newcommand{\Hts}{H^2(\Omega)^2}
\newcommand{\Htc}{H^2(\Omega)^3}
\newcommand{\Htz}{H^0(\Omega)}
\newcommand{\Hh}{H^{1/2}(\Gamma)}
\newcommand{\Hhi}{H^{1/2}(\Gamma_i)}
\newcommand{\Hmh}{H^{-1/2}(\Gamma)}
\newcommand{\Hdiv}{H(\divvr;\,\Omega)}
\newcommand{\Hdivh}{H(\divv;\,\wh \Omega)}
\newcommand{\hcurl}{H(\curl\,A;\,\Omega)}
\newcommand{\Hcurl}{H(\curll\,A;\,\Omega)}
\newcommand{\Hcrl}{H(\curll\,;\,\Omega)}
\newcommand{\hcrl}{H(\curl\,;\,\Omega)}
\newcommand{\Hcrlh}{H(\curll\,;\,\wh\Omega)}
\newcommand{\hcrlh}{H(\curl\,;\,\wh\Omega)}
\newcommand{\Wdiv}{\BW_0(\mbox{\divv}\,;\,\Omega)}
\newcommand{\Wcurl}{\BW_0(\mbox{\curl}\,A;\,\Omega)}
\newcommand{\WcrossV}{\BW \times V}

\def\grad{{\nabla}}

\def\calS{{\cal S}}
\def\calT{{\cal T}}
\def\cA{{\mathcal A}}
\def\cB{{\cal B}}
\def\cH{{\cal H}}
\def\ba{{\mathbf{a}}}
\def\cM{{\cal M}}
\def\cN{{\cal N}}
\def\cT{{\mathcal{T}}}
\def\cE{{\mathcal{E}}}

\newcommand{\aA}{{ \a_{\sF,_A}}}

\newcommand{\aH}{{ \a_{\sF,_H}}}

\def\tsigma{{\tilde{\tilde{\sigma}}}}

\def\Om{\Omega}

\newcommand{\sC}{{\mathbb C}}

\newcommand{\norm}[1]{\left\Vert#1\right\Vert}
\def\vvvert{|\kern-1pt|\kern-1pt|}
\newcommand{\newnorm}[1]{\vvvert #1\vvvert}

\def\bE{{\bf E}}
\def\bS{{\bf S}}
\def\br{{\bf r}}
\def\bW{{\bf W}}
\def\bLambda{{\bf \Lambda}}

\newcommand{\lJump}{[\![}
\newcommand{\rJump}{]\!]}
\newcommand{\jump}[1]{[\![ #1]\!]}

\newcommand{\sd}{\bsigma^{\Delta}}
\newcommand{\rd}{\brho^{\Delta}}

\newcommand{\eps}{\epsilon}

\newcommand{\dd}{\underline{{\mathbf d}}}
\newcommand{\C}{\rm I\kern-.5emC}
\newcommand{\R}{\rm I\kern-.19emR}
\newcommand{\W}{{\mathbf W}}
\def\3bar{{|\hspace{-.02in}|\hspace{-.02in}|}}
\newcommand{\A}{{\mathcal A}}

\def\X{{\mathbb{X}}}

\newcommand{\BNB}{{Banach-Nec\v{a}s-Babu\v{s}ka }}

\title [Coerciveness Proofs of FOSLS for Elliptic PDEs]{  
Several Proofs of Coerciveness of First-Order System Least-Squares Methods for General Second-Order Elliptic PDEs}
\author{Shun Zhang}
\address{Department of Mathematics, City University of Hong Kong, Kowloon Tong, Hong Kong SAR, China}
\email{shun.zhang@cityu.edu.hk}
\date{\today}
\maketitle

\begin{abstract} In this paper, we present proofs of the coerciveness of first-order system least-squares methods for general (possibly indefinite) second-order linear elliptic PDEs under a minimal uniqueness assumption. 
For general linear second-order elliptic PDEs, the uniqueness, existence, and well-posedness are equivalent due to the compactness of the operator and Fredholm alternative. Thus only a minimal uniqueness assumption is assumed: the homogeneous equation has a unique zero solution. The coerciveness of the standard variational problem is not required. 
The paper's main contribution is our first proof, which is a straightforward and short proof using the inf-sup stability of the standard variational formulation. The proof can potentially be applied to other equations or settings once having the standard formulation's stability. 
We also present two other proofs for the least-squares methods of general second-order linear elliptic PDEs. The second proof is based on a lemma introduced in the discontinuous Petrov-Galerkin method, and the third proof is based on various stability analyses of the decomposed problems. 
As an application, we also discuss least-squares finite element methods for problems with a nonsingular $H^{-1}$ right-hand side.
\end{abstract}


\section{Introduction} 

The least-squares variational principle and the corresponding least-squares finite element methods based on a first-order system reformulation have been widely used in numerical solutions of partial differential equations, see for example \cite{CLMM:94,CMM:97,BG:98,Jiang:98,BG:09,CMM:03,CS:04,CLW:04,Ku:07,LSQoI:14,CFZ:15}. Compared to the standard variational formulation and the related finite element methods, the first-order system least-squares finite element methods have several known advantages: (1) The least-squares continuous problem is well-posed as long as the linear PDE is well-posed.  (2) The discrete problem is stable and optimal accurate as long as the discrete spaces are the subspaces of the corresponding abstract solution spaces. (3) The resulting linear systems are symmetric positive definite. Fast solvers like multigrid can be used to solve the discrete problem. (4) The least-squares functional is a good error indicator/estimator for the finite element mesh refinement and the error control. (5) The boundary condition can be easily handled in the least-squares formulation with strong or weak enforcement.

The key to establishing the well-posedness and the a priori and a posteriori error estimates of the first-order system least-squares finite element methods is the coerciveness of the corresponding bilinear forms. The simplest way uses the least-squares graph norm as the norm of choice, then the uniqueness of the PDE can be used to show the well-posedness of the least-squares system. Such analysis is too crude and often only be used for least-squares bilinear forms we have less understanding of, see \cite{LZ:18,LZ:19,QZ:20}.

For problems like linear second-order elliptic PDEs, a more refined analysis is often needed to establish the norm equivalence: the equivalence of the least-squares norm and the standard Sobolev norms of the unknowns. For $L^2$-based first-order system least-squares methods for second-order elliptic equations with a new variable flux in $H(\divvr)$ space, the norm equivalence of the least-squares norm and the $H(\divvr)\times H^1$-norm is needed. For the $H^{-1}$-based first-order system least-squares method, the norm equivalence of the least-squares norm and the $L^2(\O)^d \times H^1$-norm is needed. The continuity of the least-squares bilinear form is often easy to check, so the main task is to prove the coerciveness of the least-squares bilinear form in the corresponding norms. For elliptic problems with a coercive standard variational formulation, the coerciveness proof is often quite simple with the help of the Poincar\'e inequality. The proof is subtler for general elliptic equations with possible indefinite weak problems.

\subsection{Review of past coerciveness proofs of least-squares methods for general second-order linear elliptic PDEs}
For the general second-order elliptic equation, there are various coerciveness proofs available in the past thirty years for first-order system least-squares methods. In \cite{CLMM:94}, the main assumption is the uniqueness of the solution of the linear elliptic equation, which is a very general setting that includes a real Helmholtz problem (whose standard variational formulation is naturally indefinite). The coerciveness is established for the $L^2$-based first-order system least-squares method. In the proof of \cite{CLMM:94}, the key tool is a norm equivalence ((3.5) of \cite{CLMM:94}) based on the bijectivity of the linear elliptic operator. Later in \cite{BLP:97}, in the same general setting,  the uniqueness of the elliptic PDE, the coerciveness of the $H^{-1}$-based first-order system least-squares method is established based on a compactness argument. In Cai's lecture notes \cite{Cai:04}, based on the \cite{BLP:97}'s arguments, the coerciveness of the $L^2$-based first-order system least-squares method is established in the same general setting, see also \cite{CFZ:15}. In \cite{Ku:07}, the coerciveness is established based on the a priori stability estimate of the general elliptic PDE. For general linear second-order elliptic PDEs, the uniqueness, existence, and well-posedness are equivalent due to the compactness of the operator and Fredholm alternative, so these three proofs are essentially equivalent. One possible shortcoming of \cite{BLP:97,Cai:04,CFZ:15} (also  the proofs of \cite{CLMM:94,Ku:07}) is that it is a proof by contradiction. When applied to problems with mesh-dependent settings, the indirect proof can not show that the coercivity constant is independent of the mesh.

In a more restrictive setting, where some assumptions on the coefficients are assumed to ensure the standard variational formulation is coercive, the coerciveness of the $L^2$-based first-order system least-squares method was proved in \cite{PCV:96}. Such a setting excludes problems like the real Helmholtz problems. 

In recent years, several new ideas emerged to prove the coerciveness of first-order system least-squares methods. In \cite{CQ:17}, based on a priori estimates of two separated problems with the terms of the least-squares functionals as right-hand sides, the coerciveness is established for the Helmholtz equation in a complex setting. 
The same argument is used in \cite{BM:20} for the self-adjoint diffusion-reaction problems. In \cite{CFLQ:14}, a new proof technique based on a lemma introduced in the discontinuous Petrov-Galerkin method \cite{GQ:14} is used to prove the coerciveness of specially designed least-squares finite element methods with weakly enforced boundary conditions. 

The analysis for finite element approximations to the possible indefinite general second-order elliptic equations is more complicated than the coercive cases. In \cite{Sch:74,SW:96}, the conforming finite element approximation of the general elliptic equations are discussed. Recently, in \cite{CL:99,CDNP:16,CNP:22}, nonconforming and mixed finite element approximations are discussed. All the results require some regularity and the mesh size of the discretization is small enough. Compared to the conforming, nonconforming, and mixed methods, the least-square finite element method for the general elliptic equations has an extra advantage: as long as the finite element spaces are conforming subspaces of the abstract solution spaces, the discrete problem is coercive without any restriction on the mesh size.

\subsection{Contribution of this paper}
Due to the equivalence of uniqueness, existence, and well-posedness of the second-order elliptic PDE, we only assume a {\it minimal uniqueness assumption: the homogeneous equation has a unique zero solution}. 
Three proofs for the coerciveness of two formulations of the first-order system least-squares methods for general second-order elliptic PDEs under the minimal uniqueness assumption are presented. One formulation is based on the diffusion flux reformulation, and the other one is based on a total physical flux reformulation. Both the $L^2$ and $H^{-1}$ first-order system least-squares methods are discussed. 

The paper's main contribution is a very short and simple proof presented in Section 3. The central part of the proof is only two to three lines. The main component of the proof is the inf-sup stability of the standard weak formulation, which is guaranteed by the minimal assumption.  Besides, this proof has the potential to be applied to other cases and scenarios, where the standard compactness/proof of contradiction tricks in \cite{CLMM:94,BLP:97,Cai:04,CFZ:15,Ku:07} may not be applicable, but an inf-sup stable bilinear form is available. 

We also generalize the proofs in \cite{CFLQ:14,GQ:14,CQ:17,BM:20} to the general second-order elliptic equation in this paper.
The second proof is based on a lemma introduced in the discontinuous Petrov-Galerkin method \cite{GQ:14,CFLQ:14}. The third proof is based on the a priori estimate of the PDEs and the proof idea from \cite{CQ:17,BM:20}. Both proofs require the minimal uniqueness assumption only. They are more complicated than the first proof. On the other hand, different proofs provide different possible directions of generalizations. Thus they are still presented in the paper.

As an application, we discuss least-squares finite element methods for the seconder order elliptic equation with an $H^{-1}$ right-hand side in the paper. The key to setting up a first-order system  is to choose the new vector variable carefully so it is in the $H(\divvr)$ space. Contrary to the usual global regularity assumption, we derive the local optimal a priori and a posteriori error estimates in the spirit of \cite{CHZ:17,Zhang:20mixed}. In \cite{FHK:21}, least-squares methods with singular righthand sides are discussed. We do not pursue this direction in the paper.

\subsection{Outline, notations, and function spaces}
The paper is organized as follows. In the remaining part of this subsection, we present notations and the function spaces.  In Section 2 we introduce mathematical equations for the second-order scalar elliptic partial differential equations and discuss the solution's existence, uniqueness, and well-posedness. Two variants of the first-order systems are presented. Then we set up the corresponding $L^2$ and $H^{-1}$ least-squares functionals and give sufficient conditions to ensure the coerciveness to avoid repeating. In Section 3, the first simple proof based on the inf-sup stability is presented. The second and third proofs are presented in Sections 4 and 5, respectively. Section 6 discusses the least-squares finite element method for the elliptic equation with an $H^{-1}$ right-hand side.

\setcounter{equation}{0}
Let $\O$ be a bounded, open, connected subset of $\mathbb{R}^d (d = 2 \mbox{ or } 3)$ with a Lipschitz continuous boundary $\p\O$. Denote by $\bn = (\bn_1, ... , \bn_d)^T$ the outward unit vector normal to the boundary. We partition the boundary of the domain $\O$ into two open subsets $\G_D$ and $\G_N$ such that $\p\O = \overline{\G_D} \cup \overline{\G_N}$ and $\G_D\cap \G_N =\emptyset$. For simplicity, we assume that $\G_D$ is not empty (i.e., $\mbox{meas}(\G_D) \neq 0$ ) and is connected.

We use the standard notations and definitions for the Sobolev spaces $H^s(\O)^d$ for $s\ge 0$. The standard associated inner product is denoted by $(\cdot , \,
\cdot)_{s,\O}$, and their respective norms are denoted by $\|\cdot \|_{s,\O}$ and
$\|\cdot\|_{s,\partial\O}$. The notation $|\cdot|_{s,\O}$ is used for semi-norms.  (We suppress the superscript $d$ because the dependence on dimension will be clear by context. We also omit the subscript $\O$ from the inner product and norm designation when there is no risk of confusion.) For $s=0$,
$H^s(\O)^d$ coincides with $L^2(\O)^d$.  The symbols $\gradt$ and $\nabla$ stand for the divergence and gradient operators, respectively. Set $H^1_D(\Omega):=\{v\in H^1(\Omega)\, :\, v=0\,\,\mbox{on }\Gamma_D\}$ and $H^1_0(\Omega):=\{v\in H^1(\Omega)\, :\, v=0\,\,\mbox{on }\p\O\}$.
Define $H(\divvr;\O)=\{\btau\in L^2(\O)^d\, :\,\gradt\btau\in L^2(\O)\}$, which is a Hilbert space under the norm $\|\btau\|_{H(\divvr;\,\O)}=\left(\|\btau\|_0^2+\|\gradt\btau\|_0^2 \right)^{1/2},
$
and define its subspace
$H_N(\divvr;\O)=\{\btau\in H(\divvr;\O)\, :\,\bn\cdot\btau=0 \,\,\mbox{on}\,\,\Gamma_N\}$. Since most of our discussion will be in $H_N(\divvr;\O)\times H_D^1(\O)$, we introduce the notation for simplicity,
\beq
\X:= H_N(\divvr;\O)\times H_D^1(\O).
\eeq
The space $(H^1_D(\O))'$ is defined as the dual space of $H_D^{1}(\O)$ and consists of the functionals $v$ for which the norm
\beq \label{def-1norm}
\|v\|_{-1} = \sup_{w\in H_D^1(\O)}\dfrac{\langle v,w\rangle_{(H^1_D(\O))'\times H_D^{1}(\O)}}{\|\nabla w\|_0}
\eeq
is finite.   When $\p\O=\Gamma_D$, then it is the usual $H^{-1}(\O)$; see for example, discussions in \cite{Evans:10,Brezis:11}. We use $\|\nabla \cdot\|_0$ instead of the standard $\| \cdot\|_1$ norm due to the Poincar\'{e} inequality.  
The space $(H^1_D(\O))'$ is a Hilbert space.
Let $S: (H^1_D(\O))'\rightarrow H_D^{1}(\O)$ to be the solution operator of the following problem:
\beq \label{Poisson}
z\in H_D^{1}(\O) \quad\mbox{with}\quad (\nabla z, \nabla v) = (f, v) \quad \forall v\in H_D^{1}(\O).
\eeq
That is, for $f\in  (H^1_D(\O))'$, $Sf = z\in H_D^{1}(\O)$ is the solution of \eqref{Poisson}. Again, we modify the setting in \cite{BLP:97} and only use the Poisson problem in its weak form.  In \cite{BLP:97}, the PDE version of $-\Delta u +u =f$ with boundary conditions is used. It is easy to see that the two settings are equivalent.
 We have the following result (Lemma 2.1 of \cite{BLP:97}): The inner product on $(H^1_D(\O))'\times (H^1_D(\O))'$ is given by 
\beq \label{inner-1}
(v, Sq) \quad \forall v,q \in (H^1_D(\O))'.
\eeq
We have $\|q\|_{-1}^2 = (q,Sq)$, for $q\in (H^1_D(\O))'$. 

Let $f=S^{-1}z$ and $v=z$ in \eqref{Poisson}, we have
\beq \label{wminusone}
\|\nabla z\|_0^2 = (S^{-1}z,z), \quad \forall z\in H^1_D(\O).
\eeq
Thus, for  $v\in  H^{1}_D(\O)$ and   $w = S^{-1}v\in  (H^1_D(\O))'$, we have
\beq \label{wv}
\|w\|_{-1}^2 = (w, Sw) = (S^{-1}v, SS^{-1}v) =(S^{-1}v, v) = \|\nabla v\|_0^2.
\eeq

By the Poincar\'e inequality, we also have the following result:
Assume $f\in L^2(\O)$, then
\beq\label{-1les0}
\|f\|_{-1} =\sup_{w\in H_D^1(\O)}\dfrac{(f,v)}{\|\nabla v\|_0} \leq C\sup_{w\in H_D^1(\O)}\dfrac{(f,v)}{\|v\|_0} \leq C\|f\|_0.
\eeq
For  $\btau \in H_N(\divvr;\O)$, we also have
\beq\label{-1les02}
\|\gradt\btau\|_{-1} =\sup_{w\in H_D^1(\O)}\dfrac{(\gradt\btau,v)}{\|\nabla v\|_0} = \sup_{w\in H_D^1(\O)}\dfrac{(\btau,\nabla v)}{\|\nabla v\|_0} \leq \|\btau\|_0.
\eeq


\section{General second-order elliptic equations and their least-squares methods}
\setcounter{equation}{0}

\subsection{General second-order elliptic equations}
For $f\in L^2(\O)$ is a given scalar function, consider the general second-order elliptic equation in divergence form
\beq \label{pde1}
\begin{array}{rcl}
-\gradt(A \nabla u) +\bb\cdot \nabla  u + c u &=&f  \mbox{ in } \O, \\[1mm]
u &=& 0 \mbox{ on } \Gamma_D, \\[1mm]
A \nabla u \cdot\bn &=& 0 \mbox{ on } \Gamma_N,
\end{array}
\eeq
and its adjoint problem (equation in physical form)
\beq \label{pde2}
\begin{array}{rcl}
-\gradt(A \nabla u+\bb u) + c u &=&f  \mbox{ in } \O, \\[1mm]
u &=& 0 \mbox{ on } \Gamma_D, \\[1mm]
(A \nabla u +\bb u)\cdot\bn &=& 0 \mbox{ on } \Gamma_N.
\end{array}
\eeq
We assume the following very mild conditions on the domain and coefficients.
\begin{assumption}\label{ass_dc}
The domain $\O$ is a bounded, open, connected subset of $\mathbb{R}^d (d = 2 \mbox{ or } 3)$ with a Lipschitz continuous boundary $\p\O$.

The diffusion coefficient matrix $A \in L^{\infty}(\O)^{d\times d}$ is a given $d\times d$ tensor-valued function;  the matrix $A$ is uniformly symmetric positive definite: there exist positive constants $0 < \Lambda_0 \leq \Lambda_1$ such that
\beq\label{A}
\Lambda_0 \by^T\by \leq \by^T A \by \leq \Lambda_1 \by^T\by
\eeq
for all $\by\in \mathbb{R}^d$ and almost all $x\in \O$. 
The coefficients $\bb \in L^{\infty}(\O)^d$ and  $c\in L^{\infty}(\O)$ are given vector- and scalar-valued bounded functions, respectively.
\end{assumption}
For simplicity, we define
\beq\label{Yst}
Xv = \bb\cdot \nabla + c v.
\eeq

Define the following bilinear form:
\beq
a(w,v) = (A\nabla w,\nabla v) +(X w,v) \quad w,v \in H^1_D(\O).
\eeq
It is easy to check that the bilinear form is continuous due to Assumption \ref{ass_dc}:
\beq
a(w,v) \leq C\|\nabla v\|_0 \|\nabla w\|_0 \quad v\in  H^1_D(\O) \mbox{ and } w\in  H^1_D(\O). 
\eeq

\begin{rem}
When $\bb=0$ and $c = -k^2$ for some $k>0$, the equation is a real Helmholtz equation. We have $a(v,v) = \|A^{1/2}\nabla v\|_0^2 -\|kv\|_0^2$, which is indefinite for a large $k$. However, as long as $k^2$ square is not an eigenvalue of $(A\nabla v,\nabla w)$, the equation still has a unique solution.
\end{rem}


\subsection{Existence, uniqueness, and well-posedness of the linear second-order elliptic PDE solution}
This subsection discusses the theories on the existence, uniqueness, and well-posedness of the second-order linear elliptic PDE. In order to handle the $H^{-1}$ least-squares formulations, we discuss the two problems with $(H^1_D(\O))'$ righthand sides. 

For a $g\in (H^1_D(\O))'$, assume that $u\in H^1_D(\O)$ is the solution of the weak problem of the equation in divergence form:
\beq\label{vp11}
\mbox{Find } u\in H^1_D(\O), \mbox{ such that }
a(u,v) = \langle g, v\rangle_{(H^1_D(\O))'\times H^1_D(\O)} \quad\forall v\in H^1_D(\O),
\eeq
or the adjoint weak problem of the equation in physical form,
\beq\label{vp22}
\mbox{Find } u\in H^1_D(\O), \mbox{ such that }
a(v,u) = \langle g, v\rangle_{(H^1_D(\O))'\times H^1_D(\O)}  \quad\forall v\in H^1_D(\O).
\eeq
For a possible indefinite linear second-order elliptic equation, the solution's existence and uniqueness theory is based on the Fredholm alternative. Since we assume the ellipticity of the PDEs (conditions \eqref{A} on  $A$), the operators associated with the divergence form and adjoint physical form problems are Fredholm operators of index zero. The uniqueness, existence, and well-posedness are equivalent. 
We collect the existence and uniqueness results of linear second-order elliptic equations in the following theorem. These results are standard and are known to the experts. We present here only for completeness and clarifications.

\begin{thm}\label{assp}
Assuming that Assumption \ref{ass_dc} is true, the following assumptions are equivalent:
\begin{enumerate}[(1)]

\item The homogeneous equation associated to \eqref{vp11}, i.e., $a(u,v) = 0$, for all $v\in H^1_D(\O)$, has $u=0$ as its unique solution.

\item 
The weak problem \eqref{vp11} has the following stability bound:
\beq \label{apriori1}
	\|\nabla u\|_0 \leq C\|g\|_{-1} \quad \forall g\in (H^1_D(\O))'.
\eeq

\item The weak problem \eqref{vp11} has a unique solution $u\in H^1_D(\O)$ for any $g\in (H^1_D(\O))'$. 

\item The homogeneous equation associated to \eqref{vp22}, i.e., $a(v,u) = 0$, for all $v\in H^1_D(\O)$, has $u=0$ as its unique solution.

\item 
The adjoint weak problem \eqref{vp22} has the following stability bound:
\beq \label{apriori2}
	\|\nabla u\|_0 \leq C\|g\|_{-1} \quad \forall g\in (H^1_D(\O))'.
\eeq

\item The adjoint weak problem \eqref{vp22} has a unique solution $u\in H^1_D(\O)$ for any $g\in (H^1_D(\O))'$.


\end{enumerate}
\end{thm}
We sketch two proofs.
\begin{proof} (I)
By the Fredholm alternative, the solution spaces of the two adjoint homogeneous problems have the same finite number of linearly independent solutions, see page 200 of \cite{BJS:64} or Theorem 5.11 of \cite{GT:01}. If one of them has only $0$ as its solution, so is the other, thus (1) $\Longleftrightarrow$ (4). 

From page 200  of \cite{BJS:64},  if (4) is assumed, then for any $g\in (H^1_D(\O))'$, the weak problem \eqref{vp11}  has a unique solution in $H_D^1(\O)$ since the requirement of the existence of a solution by the Fredholm alternative is that $\langle g, v\rangle =0$ for any $v\in H_D^1(\O)$ being the solution of homogeneous adjoint equation \eqref{vp22}, which is $0$. Then we get (4) $\Longrightarrow$ (3). Similarly, we get (1) $\Longrightarrow$ (6).

Again from the Fredholm alternative, we get (1) $\Longrightarrow$ (2) and 
(4) $\Longrightarrow$ (5), the a priori estimate.

On the other hand, if (2) is true, we immediately have (1) by letting $g=0$. Similarly, (5) $\Longrightarrow$ (4).

The conclusions (3) $\Longrightarrow$ (1) and (6) $\Longrightarrow$ (4) are obvious. Thus, we have the equivalence of the assumptions in the theorem.
%
%
\end{proof}

\begin{proof} (II)
The equivalence of (1)  $\Longleftrightarrow$ (2)  $\Longleftrightarrow$ (3) can also be proved directly without the help of the adjoint problem. Following the proof of Theorem 8.3 of \cite{GT:01}, we get (1) $\Longrightarrow$ (3). (1) $\Longrightarrow$ (2) is a result of the Fredholm alternative. (2) $\Longrightarrow$ (1)  and (3) $\Longrightarrow$ (1) are obvious. Similarly, we have (4)  $\Longleftrightarrow$ (5)  $\Longleftrightarrow$ (6). We only need to show the equivalence of (1-3) and (4-6).

From the \BNB theory \cite{Babuska:71,Braess:07,XZ:03}, the problem \eqref{vp11}, which has a continuous bilinear form, is well-posed if and only  the following conditions hold:
\beq\label{infsup}
\inf_{v\in H_D^1(\O)}\sup_{w\in H_D^1(\O)} \dfrac{a(v,w)}{\|\nabla v\|_0 \|\nabla w\|_0} >0, \quad
\inf_{w\in H_D^1(\O)}\sup_{v\in H_D^1(\O)} \dfrac{a(v,w)}{\|\nabla v\|_0 \|\nabla w\|_0} >0.
\eeq
furthermore if \eqref{infsup} holds, then
\beq\label{bnb}
\inf_{v\in H_D^1(\O)}\sup_{w\in H_D^1(\O)} \dfrac{a(v,w)}{\|\nabla v\|_0 \|\nabla w\|_0} =
\inf_{w\in H_D^1(\O)}\sup_{v\in H_D^1(\O)} \dfrac{a(v,w)}{\|\nabla v\|_0 \|\nabla w\|_0} = \beta >0.
\eeq
Thus, we immediately get \eqref{vp22} is also well-posed and thus 
(1-3) $\Longleftrightarrow$ (4-6). The proof of the opposite direction is identical.
\end{proof}

\begin{rem}
The inf-sup condition \eqref{infsup} or \eqref{bnb} is also equivalent to
\beq\label{is}
\beta \|\nabla v\|_0 \leq \sup_{w\in H_D^1(\O)} \dfrac{a(v,w)}{\|\nabla w\|_0}
\quad\mbox{and}\quad
\beta \|\nabla v\|_0 \leq \sup_{w\in H_D^1(\O)} \dfrac{a(w,v)}{\|\nabla w\|_0}
\quad \forall v\in H_D^1(\O),
\eeq
for some $\beta >0$.
These conditions are also equivalent to the stability estimates \eqref{apriori1} and \eqref{apriori2}.
\end{rem}

\begin{rem}
The \BNB theory is established for general linear problems in the Banach/Hilbert spaces setting, which states that the problem is well-posed if and only if the bilinear form is continuous and satisfies the inf-sup conditions \eqref{infsup}. For second-order elliptic equations, due to the compactness and Fredholm alternative, we only need one condition from Theorem \ref{assp}.
\end{rem}

\begin{rem}
The stability constant $\beta$ in \eqref{bnb} is mesh-independent, but it can be very close to zero. For example, for the Helmholtz equation with $\bb=0$ and $-c$ is very close to the eigenvalue of $(A\nabla v,\nabla w)$.
\end{rem}
\begin{rem}
A proof of (1) $\Longrightarrow$ (3) and (4) $\Longrightarrow$ (6) can be found in  Lemma 2.2 of \cite{BLP:98} for some more general mixed boundary conditions with a compactness argument.
\end{rem}

\subsection{Least-squares minimization problems}
For the elliptic equation in the divergence form \eqref{pde1}, let the flux $\bsigma = -A \nabla u$.
We have the  first-order system:
\beq \label{fosys1}
\left\{
\begin{array}{rclll}
\bsigma + A \nabla u &=&0& \mbox{ in } \O,
\\[1mm]
\gradt\bsigma +Xu &=&f&  \mbox{ in } \O,\\
u &=& 0 &\mbox{ on } \Gamma_D, \\
\bsigma \cdot\bn &=& 0& \mbox{ on } \Gamma_N.
\end{array}
\right.
\eeq
For $u\in H^1_D(\O)$, we have  $\bsigma = -A \nabla u \in L^2(\O)^d$ and $\gradt\bsigma  = f- Xu\in L^2(\O)$, so $(\bsigma,u)\in \X = H_N(\divvr;\O)\times H_D^1(\O) $.
For $(\btau,v)\in \X$, define least-squares functionals for the system \eqref{fosys1},
\begin{eqnarray}
L(\btau,v;f) &:=& \|A^{-1/2}\btau + A^{1/2}\nabla v\|_0^2 +\|\gradt\btau +Xv-f \|_0^2,\\[1mm]
L_{-1}(\btau,v;f) &:=& \|A^{-1/2}\btau + A^{1/2}\nabla v\|_0^2 +\|\gradt\btau +Xv-f \|_{-1}^2.
\end{eqnarray}
The corresponding least-squares minimization problems are:
\begin{eqnarray}
\mbox{Find } (\bsigma,u)\in \X  &\mbox{ s.t. }&
L(\bsigma,u;f) = \inf_{(\btau,v)\in \X}L(\btau,v;f) \\
\mbox{and find } (\bsigma,u)\in \X  &\mbox{ s.t. }&
L_{-1}(\bsigma,u;f) = \inf_{(\btau,v)\in \X}L_{-1}(\btau,v;f). 
\end{eqnarray}
For the elliptic equation in the physical form \eqref{pde2}, define the total flux as $\bsigma = -A \nabla u-\bb u$, then
\beq \label{fosys2}
\left\{
\begin{array}{rllll}
\bsigma +A \nabla u+\bb u&=&0& \mbox{ in } \O,\\[1mm]
\gradt\bsigma + c u &=&f&  \mbox{ in } \O,\\
u &=& 0 &\mbox{ on } \Gamma_D, \\
\bsigma \cdot\bn &=& 0& \mbox{ on } \Gamma_N.
\end{array}
\right.
\eeq
Similarly, we have  $(\bsigma,u)\in \X$.
For $(\btau,v)\in \X$, define least-squares functionals for the system \eqref{fosys2},
\begin{eqnarray}
J(\btau,v;f) &:=& \|\btau + A\nabla v+\bb v\|_0^2 +\|\gradt\btau + c v-f \|_0^2,\\
J_{-1}(\btau,v;f) &:=& \|\btau + A\nabla v+\bb v\|_0^2 +\|\gradt\btau + c v-f \|_{-1}^2.
\end{eqnarray}

\begin{rem}
We discuss the problem  \eqref{pde2} independently since the total physical flux is different from the flux \eqref{pde1} and has many applications. The introduction of the total flux in the least-squares context can be found in \cite{LTV:97}.
\end{rem}

For any $(\btau,v)\in L^2(\O)^d\times H^1_D(\O)$ and $(\brho,w) \in \X$,
define the following two norms, respectively: 
\beq
\|(\btau,v)\|^2 := \|\btau\|_0^2+\|\nabla v\|_0^2 \quad\mbox{and}\quad
\tri(\brho,w)\tri^2 := \|\brho\|_0^2+\|\gradt \brho\|_0^2+ \|\nabla w\|_0^2.
\eeq
\begin{rem}
By the Poincar\'e inequality, we can also replace $\|\nabla v\|_0$ by a standard  $\|v\|_1$ in the above norms. We use the norm  $\|\nabla v\|_0$ for simplicity.
\end{rem}

From the Cauchy-Schwarz, triangle, and Poincar\'e inequalities and the bounds of the coefficients $(A, \bb$, and $c)$ of the underlying problems (Assumption \ref{ass_dc}), we immediately have the following upper bounds: There exists $C>0$, such that, 
\begin{eqnarray}
L(\btau,v;0)  \leq C\tri(\btau,v)\tri^2 \quad \mbox{and}\quad
J(\btau,v;0)  \leq  C\tri(\btau,v)\tri^2, \quad \forall (\btau,v)\in \X.
\end{eqnarray}
Together with \eqref{-1les0} and \eqref{-1les02}, we can show that there exists $C>0$, such that,
\begin{eqnarray}
L_{-1}(\btau,v;0)  \leq C\|(\btau,v)\|^2 \quad \mbox{and}\quad
J_{-1}(\btau,v;0)  \leq  C\|(\btau,v)\|^2, \quad \forall (\btau,v)\in \X.
\end{eqnarray}

In this paper, we want to prove the following coerciveness results: For all $(\btau,v)\in \X$, there exists $C>0$, such that,
\begin{eqnarray}\label{cov_L}
 C\tri(\btau,v)\tri^2 &\leq & L(\btau,v;0), \quad 
C\|(\btau,v)\|^2  \leq   L_{-1}(\btau,v;0), \\ \label{cov_J}
C\tri(\btau,v)\tri^2  &\leq &  J(\btau,v;0),\quad 
C\|(\btau,v)\|^2  \leq   J_{-1}(\btau,v;0). 
\end{eqnarray}
\begin{rem}
If the coercive results \eqref{cov_L} to \eqref{cov_J} hold, then we have the following equivalence:
\begin{eqnarray}
C_1 L(\btau,v;0) \leq J(\btau,v;0) \leq C_2 L(\btau,v;0) \quad (\btau,v)\in \X;\\
C_4 L_{-1}(\btau,v;0) \leq J_{-1}(\btau,v;0) \leq C_4 L_{-1}(\btau,v;0) \quad (\btau,v)\in \X.
\end{eqnarray}
\end{rem}
Before discussing the complete coerciveness proofs, we introduce some conditions that ensure coerciveness. These conditions are elementary and well-known. We list them here to avoid repeating them in future proofs. 
\begin{thm} \label{thm_part}
For  $(\btau,v)\in \X  $,  we have
\begin{eqnarray}\label{c_cov_L}
 C\|\nabla v\|_0^2 \leq  L(\btau,v;0) &\Longrightarrow&
 C\tri(\btau,v)\tri^2 \leq  L(\btau,v;0), \\ \label{c_cov_L-1}
  C\|\nabla v\|_0^2 \leq  L_{-1}(\btau,v;0) &\Longrightarrow& 
C\|(\btau,v)\|^2  \leq   L_{-1}(\btau,v;0), \\ \label{c_cov_J}
 C\|\nabla v\|_0^2 \leq  J(\btau,v;0) &\Longrightarrow&
C\tri(\btau,v)\tri^2  \leq   J(\btau,v;0), \\ \label{c_cov_J-1}
 C\|\nabla v\|_0^2 \leq  J_{-1}(\btau,v;0) &\Longrightarrow&
C\|(\btau,v)\|_0^2  \leq   J_{-1}(\btau,v;0). 
\end{eqnarray}
\end{thm}
\begin{proof}
From the triangle and Poincar\'e inequalities, we have
\begin{eqnarray*}
\|\gradt \btau\|_0 &\leq& \|\gradt \btau + Xv\|_0 + \|Xv\|_0 
\leq  \|\gradt \btau + Xv\|_0 + C\|\nabla v\|_0+ C\|\btau\|_0 \\
\mbox{and   }\|\btau\|_0 &\leq&  \|\btau +A\nabla v\|_0 + C\|\nabla v\|_0.
\end{eqnarray*}
Thus, if the first inequality of \eqref{c_cov_L} is true, we have the coerciveness. The rest results in the theorem can be proved similarly.
\end{proof}

\section{Proof I: A Simple Direct Proof from the Inf-Sup Stability}
\setcounter{equation}{0}
In this section, we present a simple proof of the coerciveness of the least-squares functionals.
\begin{thm}
Assuming one of the conditions of Theorem \ref{assp} is true, the coercive results in \eqref{cov_L} and \eqref{cov_J} hold.
\end{thm}
\begin{proof}
By the integration by parts, we have $(\btau, \nabla w)+(\gradt\btau,w) =0$ for any $(\btau,w) \in H_N(\divvr;\O)\times H_D^1(\O)$. Then $\btau \in H_N(\divvr;\O)$ and $v$ and $w$ in $H_D^1(\O)$, we have
\begin{eqnarray} \label{vww}
a(v,w) &=& (A\nabla v, \nabla w)+(Xv,w) =  (A\nabla v+\btau, \nabla w)+(\gradt\btau+Xv,w),\\ \label{wvv}
a(w,v) &=& (A\nabla v+\bb v, \nabla w)+(cv,w) \\\nonumber &=&  (A\nabla v+\bb v+\btau, \nabla w)+(\gradt\btau+cv,w).
\end{eqnarray}
It follows from \eqref{is}, \eqref{vww}, the Cauchy-Schwarz and Poincar\'e inequalities, the definition of the minus one norm, the assumption on $A$, and \eqref{-1les0}, for any $(\btau,v) \in H_N(\divvr;\O)\times H_D^1(\O)$, 
\begin{eqnarray*}
\beta \|\nabla v\|_0 &\leq& \sup_{w\in H_D^1(\O)} \dfrac{a(v,w)}{\|\nabla w\|_0} 
=\sup_{w\in H_D^1(\O)}\dfrac{ (A\nabla v+\btau, \nabla w)+(\gradt\btau+Xv,w)}{\|\nabla w\|_0} \\
&\leq & C\|A^{1/2}\nabla v+A^{-1/2}\btau\|_0+\|\gradt\btau + Xv\|_{-1} \\
&\leq & C (\|A^{1/2}\nabla v+A^{-1/2}\btau\|_0 + \|\gradt\btau +  Xv\|_0) .
\end{eqnarray*}
Similarly, using \eqref{wvv}, for any $(\btau,v) \in H_N(\divvr;\O)\times H_D^1(\O)$, we have
\begin{eqnarray*}
\beta \|\nabla v\|_0 &\leq& \sup_{w\in H_D^1(\O)} \dfrac{a(w,v)}{\|\nabla w\|_0} 
=\sup_{w\in H_D^1(\O)}\dfrac{ (A\nabla v+\bb v+\btau, \nabla w)+(\gradt\btau+cv,w)}{\|\nabla w\|_0} \\
&\leq & C\|A\nabla v+\bb v+ \btau\|_0+\|\gradt\btau + cv\|_{-1} \\
&\leq & C (\|A\nabla v+\bb v+ \btau\|_0 + \|\gradt\btau +  cv\|_0) .
\end{eqnarray*}
Thus the conditions in Theorem \ref{thm_part} are satisfied, we have the coerciveness.
\end{proof}

\begin{rem}
Compared to the complicated proofs in \cite{CLMM:94,BLP:97,Cai:04,CFZ:15,Ku:07}, the proof in this section is dramatically simplified. We hide the compactness argument used in \cite{CLMM:94,BLP:97,Cai:04,CFZ:15} to the stability of the original bilinear form. For the second-order elliptic equation, the proof of the stability relies on the Fredholm alternative, which also relies on compactness; thus, the proofs are essentially equivalent. The advantage of the current proof is that it can be applied to other cases where the stability of the bilinear form is available, and the stability is not necessarily related to the compactness. For example, for an $L^2$ div least-squares method with Crouzeix-Raviart finite element approximation for the general second-order elliptic equation, we have the discrete stability of the bilinear form under the condition the mesh is fine enough \cite{CDNP:16}. We then can use the above proof method to derive the discrete coerciveness of the $L^2$ div least-squares method with Crouzeix-Raviart finite element approximation, see \cite{LZ:22}. On the other hand, the compactness arguments in  \cite{BLP:97,Cai:04,CFZ:15,Ku:07} do not work in this discrete scenario since the compactness argument is a proof of contradiction and can not be used in mesh-dependent methods to prove that the coercive constant is independent of the mesh size.

As seen from the proof, the proof also guides us in choosing the right weight of different terms.
\end{rem}

\section{Proof II: Proofs Based on an Equivalence Theorem and Inf-Sup Stability}
\setcounter{equation}{0}

The primary tool in the proof of this section is an equivalence theorem from the discontinuous Petrov-Galerkin method \cite{GQ:14}. Theorem \ref{eqv} can be found in Theorem 2.1 of \cite{GQ:14} and Theorem 3.1 of \cite{CFLQ:14}. The first two equalities in \eqref{eq_eqv} are trivial, and the proof of the last equality in \eqref{eq_eqv} is based on the simple fact that the last two suprema in the theorem are achieved for $v=Tw$. 

\begin{thm}{\bf An Equivalence Theorem.}\label{eqv}
Let $U$ be a normed linear space, and $V$ be a Hilbert space with associate bilinear form $(\cdot,\cdot)_V$ and norm $\|\phi\|_V$. The operator $T$ is a linear map from $U$ to $V$. Then, for any $\psi\in U$
\beq \label{eq_eqv}
\|T\psi\|_V =
\sup_{\chi\in U} \dfrac{(T \psi, T\chi)_V}{\|T \chi\|_V} =
\sup_{\phi\in T(U)} \dfrac{(T \psi, \phi)_V}{\|\phi\|_V} = 
\sup_{\phi\in V} \dfrac{(T \psi, \phi)_V}{\|\phi\|_V}.
\eeq
\end{thm}
In most first-order least-squares methods, we want to prove 
$$
C\|\psi\|_U \leq \|T\psi\|_V \quad \forall \psi\in U,
$$  
for the bilinear form $(T\psi, T\chi)_V$ with $C>0$ (or some simpler condition, see for example, Theorem \ref{thm_part}). Theorem \ref{eqv} enables us to choose the test function in a larger space $V$ instead of $\chi\in U$ or $\phi\in T(U)$. 
In most $L^2$-based first-order least-squares methods, $V$ is a simple $L^2$-space on some domain or its boundary. 

\begin{thm} \label{thmbb}
Assuming one of the conditions of Theorem \ref{assp} is true, the coercive results in \eqref{cov_L} hold.
\end{thm}
\begin{proof}
Let $U_0=\X$ and  $V_0 = L^2(\O)^d \times L^2(\O)$. 
Let $U_{-1}=\{\btau \in L^2(\O)^d: \btau\cdot\bn = 0 \mbox{ on }\Gamma_N \}\times H^1_D(\O)$ and let $V_{-1} = L^2(\O)^d \times (H^1_D(\O))'$.

The operator $T$ for both from $U_0$ to $V_0$ and $U_{-1}$ to $V_{-1}$ is defined as:
\beq \label{T}
T(\btau,v) = \left(
\begin{array}{ccc}
A^{-1/2}\btau+ A^{1/2}\nabla v \\
\gradt \btau + Xv
\end{array}
\right).
\eeq
Then we have
\begin{eqnarray}
\|T(\btau,v) \|_{V_0}^2 &=& \|A^{-1/2}\btau+ A^{1/2}\nabla v \|_0^2 + \|\gradt \btau + Xv \|_0^2, \\
\|T(\btau,v) \|_{V_{-1}}^2 &=& \|A^{-1/2}\btau+ A^{1/2}\nabla v \|_0^2 + \|\gradt \btau + Xv \|_{-1}^2.
\end{eqnarray}
Let $w_0\in H_D^1(\O) \subset L^2(\O)$ and 
$\br_0 = A^{1/2}\nabla w_0 \in L^2(\O)^d$.
Then by integration by parts,
\begin{eqnarray*}
(T(\btau,v), (\br_0,w_0))_{V_0} &=& (A^{-1}\btau+ \nabla v,  A\nabla w_0)  +  
(\gradt \btau +  Xv,    w_0) = a (v,w_0).
\end{eqnarray*}
It is easy to $\|(\br_0,w_0)\|_{V_0} =\|A^{1/2}\nabla w_0\|_0 + \|w_0\|_0$  is equivalent to $\|\nabla w_0\|_0$ .
Then from Theorem \ref{eqv} and the inf-sup condition \eqref{bnb} of the \BNB theory,
$$
\|T(\btau,v) \|_{V_0} = \sup_{(\br,w) \in V_0}\dfrac{(T(\btau,v), (\br,w))_{V_0}}{\|(\br,w)\|_{V_0}}\geq  C_1  \sup_{w_0 \in  H_D^1(\O)} \dfrac{a(v,w_0)}{\|\nabla w_0\|_0} \geq C_2 \|\nabla v\|_0.
$$
The $L^2$-version coerciveness in \eqref{cov_L} is proved with the help of Theorem \ref{thm_part}.

For $z\in H_D^1(\O)$, let
$
\br_{-1} = A^{1/2}\nabla z$ and $w_{-1} = S^{-1}z$.
Since $z\in H_D^1(\O)$, thus $S^{-1}$ is well-defined.  By \eqref{wv},  $\|w_{-1}\|_{-1}=\|\nabla z\|_0$. We get
$$
\|(\br_{-1},w_{-1})\|_{V_{-1}} = \|A^{1/2}\nabla z\|_0 + \|w_{-1}\|_{-1} = \|A^{1/2}\nabla z\|_0 +\|\nabla z\|_0.
$$
Thus  $\|\nabla z\|_0$ is equivalent to $\|(\br_{-1},w_{-1})\|_{V_{-1}}$.

By the definition of the inner product in $(H^1_D(\O))'$ in \eqref{inner-1} and the definitions of $\br_{-1}$ and $w_{-1}$, we have
\begin{eqnarray*}
(T(\btau,v), (\br_{-1},w_{-1}))_{V_{-1}}& =& (A^{-1}\btau+ \nabla v,  A\nabla z) + (\gradt \btau +  Xv,    SS^{-1}z), \\
 &=&  (A^{-1}\btau+ \nabla v,  A\nabla z) + (\gradt \btau +  Xv, z)=a(v, z).
\end{eqnarray*}
Then with almost the same calculations as the $V_0$ case, we have
$\|T(\btau,v)\|_{V_{-1}} \geq C \|\nabla v\|_0$.
Thus  the conditions in Theorem \ref{thm_part} are satisfied, and the $H^{-1}$-version of coerciveness in \eqref{cov_L} is proved.
\end{proof}

By defining 
$$
T(\btau,v) = \left(
\begin{array}{ccc}
\btau+ A\nabla v +\bb v \\
\gradt \btau + c v
\end{array}
\right),
$$ 
and using very similar arguments, we can easily prove the following theorem.
\begin{thm}
Assuming one of the conditions of Theorem \ref{assp} is true, the coercive results in \eqref{cov_J} hold.
\end{thm}

\begin{rem}
The proof in this section is not necessarily simpler than the one presented in the previous section, but it gives another explanation and is probably worth keeping in mind.  The advantage of the proofs here is that we have more freedom to choose the test function to make a more refined analysis possible. In fact, in \cite{CFLQ:14}, test functions similar to those in \cite{AM:09} are chosen. 
\end{rem}

\section{Proof III:  Proofs Based on Stability Estimates of Decomposed Problems}
\setcounter{equation}{0}
We list two important cases of the a priori estimates \eqref{apriori1} and \eqref{apriori2}. The first one is that 
\beq
\langle g, v\rangle_{(H^1_D(\O))'\times H^1_D(\O)} := (f_1,v)
\eeq
for some $f_1\in L^2(\O)$. It is obvious that $f_1\in (H^1_D(\O))'$, then by \eqref{apriori1}, \eqref{apriori2}, and \eqref{-1les0}, we have 
\beq \label{aprioriL2}
	\|\nabla u\|_0 \leq C\|f_1\|_{-1}\leq C \|f_1\|_0.
\eeq
The second one is that 
\beq
\langle g, v\rangle_{(H^1_D(\O))'\times H^1_D(\O)} := (\bff_2,\nabla v)
\eeq
for some $\bff_2\in L^2(\O)^d$. Then by \eqref{apriori1}, \eqref{apriori2}, and \eqref{def-1norm}, we have 
\beq \label{apriori_vec}
	\|\nabla u\|_0 \leq C\|g\|_{-1}  = C \sup_{v\in H^1_D(\O)} \dfrac{(\bff_2,\nabla v)}{\|\nabla v\|_0} \leq C\|\bff_2\|_0.
\eeq
Note that, for the second case and $\G_D=\p\O$, we have $(H^1_D(\O))' = (H^1_0(\O))' = H^{-1}(\O)$, $g$ can be (at least formally) written as $-\gradt\bff_2$, see the discussion in \cite{Evans:10,Brezis:11}. For the general case $\G_D\neq \p\O$, we can not write $g$ as $-\gradt\bff_2$, thus the weak forms \eqref{vp11} and \eqref{vp22} instead of the strong PDE forms are preferred in our presentations. 

Before diving into the proof, we first present stability estimates for four auxiliary problems.

\noindent{\bf Problems 1 \& 2.}
Let $\bff\in L^2(\O)^d$ and $g\in L^2(\O)$. Consider the following first-order systems:
\beq \label{problem12}
\mbox{Problem 1}
\left\{
\begin{array}{rlllll}
\br + A\nabla w &=& \bff &\mbox{ in } \O,\\[1mm]
\gradt\br +Xw&=& 0  &\mbox{ in }  \O,\\[1mm]
w&=&0& \mbox{ on } \Gamma_D,\\[1mm]
\br\cdot\bn &= &0& \mbox{ on } \Gamma_N.
\end{array}
\right.
\quad\mbox{Problem 2}
\left\{
\begin{array}{rlllll}
\bs + A\nabla z &=& \bzero , &\mbox{ in } \O,\\[1mm]
\gradt\bs +Xz  &= &g  &\mbox{ in } \O,\\[1mm]
z&=&0& \mbox{ on } \Gamma_D,\\[1mm]
\bs\cdot\bn &= &0& \mbox{ on } \Gamma_N.
\end{array}
\right.
\eeq
\begin{lem}
Assuming one of the conditions of Theorem \ref{assp} is true, the following stability estimates for \eqref{problem12} are true:
\beq \label{apriori1111}
\|\nabla w\|_0 \leq C \|\bff\|_0 
\quad\mbox{and}\quad
\|\nabla z\|_0\leq C \|g\|_{-1} \leq C \|g\|_0.
\eeq
\end{lem}
\begin{proof}
First, we construct a unique solution for Problem 1. Let $w\in H^1_D(\O)$ be the unique solution of 
\beq \label{eqw}
(A\nabla w, \nabla v) +  (X w, v) =  (\bff, \nabla v),
\quad \forall v\in H^1_D(\O).
\eeq
Let $\br =\bff-A\nabla w$, which is the first equation of Problem 1, and substitute it into \eqref{eqw}. We have
\beq
-(\br,\nabla v) +(Xw,v)  =0 \quad  \forall v\in H^1_D(\O).
\eeq
Following the arguments in \cite{BM:20}, picking $v\in C_0^{\infty}(\O)$ and integrating by parts on the term $(\br, \nabla v)$, we get the second equation of Problem 1. For any $v\in H^1_D(\O)$, 
$$
\langle \br\cdot\bn, v\rangle_{H^{-1/2}(\Gamma_N) \times H^{1/2}(\Gamma_N) } =
(\br, \nabla v)+(\gradt \br, v) = (\gradt \br+Xw,v) =0,
$$
thus the boundary condition of $\br$ of Problem 1 is satisfied. Thus, we have completed our construction of the solutions of Problem 1.
For the equation \eqref{eqw}, the a priori estimate \eqref{apriori_vec} gives the bound $\|\nabla w\|_0  \leq C \|\bff\|_0$.

We then work on Problem 2. Similar to Problem 1, we construct solutions of Problem 2 first.
Let $z\in H^1_D(\O)$ be the unique solution of the following problem:
\beq \label{eqz}
(A\nabla z,\nabla v) + (X z,v) = (g,v)\quad\forall v\in H_D^1(\O).
\eeq
Let $\bs = -A\nabla z$ and substitute it into \eqref{eqz}, and use the same arguments as in Problem 1, we get the second equation and the boundary condition of $\bs$ of Problem 2.  Thus, we have constructed $z$ and $\bs$ to be the solution of Problem 2.
For the equation \eqref{eqz}, based on the assumption \eqref{aprioriL2}, we have $\|\nabla z\|_0\leq C \|g\|_{-1} \leq C \|g\|_0$.
\end{proof}
Then we can prove the coerciveness of $L$ and $L_{-1}$.
\begin{thm}\label{p12}
Assuming one of the conditions of Theorem \ref{assp} is true, the coerciveness results in \eqref{cov_L} hold.
\end{thm}
\begin{proof}
For $(\btau,v)\in \X$, in \eqref{problem12}, let $\bff = \btau + A\nabla v $ for Problems 1 and $g= \gradt\btau + Xv$ for Problem 2.
By the linearity of Problems 1 and 2, we have  $\btau = \br+\bs$ and $v= w+z$. By the a priori error estimates \eqref{apriori1111}, we have
$$
\|\nabla v\|_0\leq C(\|\btau + A\nabla v\|_0 + \| \gradt\btau + Xv\|_{-1})\leq 
C(\|\btau + A\nabla v\|_0 + \| \gradt\btau + Xv\|_0).
$$
Thus conditions of Theorem \ref{thm_part} are satisfied, the coerciveness results in \eqref{cov_L} are proved.
\end{proof}

\noindent{\bf Problems 3 \& 4.}
Let $\bff\in L^2(\O)^d$ and $g\in L^2(\O)$. Consider the following first-order systems:
\beq \label{problem34}
\mbox{Problem 3}
\left\{
\begin{array}{rllllll}
\br + A\nabla w +\bb w &=& \bff &\mbox{ in }  \O,\\[1mm]
\gradt\br + c w&=& 0  &\mbox{ in }   \O,\\[1mm]
w&=&0& \mbox{ on } \Gamma_D,\\[1mm]
\br\cdot\bn &= &0& \mbox{ on } \Gamma_N.
\end{array}
\right.
\mbox{Problem 4}
\left\{
\begin{array}{rllllll}
\bs + A\nabla z +\bb z &=&0&  \in \O,\\[1mm]
\gradt\bs + c z &=&g &\mbox{ in }  \O,\\[1mm]
z&=&0& \mbox{ on } \Gamma_D,\\[1mm]
\bs\cdot\bn &= &0& \mbox{ on } \Gamma_N.
\end{array}
\right.
\eeq
By the almost identical arguments, we can show the following estimates for Problems 3 and 4:
\beq
\|\nabla w\|_0 \leq C \|\bff\|_0 \quad\mbox{and}\quad
\|\nabla z\|_0\leq C \|g\|_{-1} \leq C \|g\|_0.
\eeq
By the very similar arguments of Theorem \ref{p12}, we can show the following theorem.
\begin{thm}
Assuming one of the conditions of Theorem \ref{assp} is true, the coerciveness results in \eqref{cov_J} hold.
\end{thm}
\begin{rem}
The proof in this section is essentially an application of the stability results of PDE to its decoupled first-order systems. Again, it looks more complicated than the proof I. Since the proof method is used in two different situations \cite{CQ:17,BM:20}, we give the proof for the general elliptic equations here for completeness. 
\end{rem}

\section{Extension to least-squares finite element methods with $H^{-1}$ right-hand side}
\setcounter{equation}{0}
This section discusses the least-squares finite element methods for the second-order elliptic equation with an $H^{-1}$ right-hand side. Such a problem appears in many situations, for example, in the goal-oriented, a posteriori error estimate \cite{IP:21} and Darcy flow.

For simplicity, we only consider the \eqref{pde1}-type of equation with a pure Dirichlet boundary condition and the $L^2$-version of the least-squares finite element method.  As pointed out in \cite{Evans:10,Brezis:11}, any functional in $H^{-1}:= (H_0^1(\O))'$ can be written as 
$f - \gradt \bg$ for $f\in L^2(\O)$ and $\bg\in L^2(\O)^d$. Thus, we consider the following problem:
\beq \label{pde1_m1}
\begin{array}{rcl}
-\gradt(A \nabla u) +\bb\cdot \nabla  u + c u &=&f_1 - \gradt (A \bff_2)  \mbox{ in } \O, \quad
u = 0 \mbox{ on } \p\O.
\end{array}
\eeq
Here  $f_1\in L^2(\O)$ and $\bff_2 \in L^2(\O)^d$ are given functions.  We have $f_1 +\gradt (A\bff_2) \in H^{-1}(\O)$. The divergence of $ A\bff_2$ should be understood in the distributional sense, i.e., 
$$
(\gradt (A\bff_2), v) = -(A\bff_2, \nabla v)\quad \forall v\in H^1_0(\O).
$$
A typical $A\bff_2$ can be $A\nabla v_h$ for $v_h$ being a function in the conforming finite element space. Such a righthand side also appears in the recovery-based error estimators, see \cite{CZ:10b}. 
The standard variational problem reads: Find $u\in H^1_0(\O)$, such that,
\beq
(A\nabla u,\nabla v)+(\bb\cdot\nabla u+c u, v) = (f_1,v)+(A\bff_2, \nabla v) \quad \forall v\in H^1_0(\O).
\eeq
For equation \eqref{pde1_m1}, the quantity $A\nabla u$ is not in $H(\divvr;\O)$. Since $\gradt(A\bff_2)$ is only in $H^{-1}(\O)$, not $L^2(\O)$, thus 
$\gradt(A\nabla u) = \bb\cdot \nabla  u + c u- f_1 - \gradt (A \bff_2)$ is only in $H^{-1}(\O)$.  Taking the simplest example, let $\bb =0$, $c=0$, $f_1=0$, $A=I$, and $\bff_2 = \nabla v_h$, where $v_h$ is a continuous piecewise linear finite element function. It is well-known that we usually do not have $\jump{\nabla v_h\cdot\bn_F}_F =0$ across an internal edge $F$ of a finite element mesh. This term often appears in the residual type of a posteriori error estimator \cite{AO:00,Ver:13}, and recovering it in the $H(\divvr)$-conforming space is the foundation of a recovery-based error estimator \cite{CZ:09}. Then for this simplest example, $A\nabla u=-A\nabla v_h \not\in H(\divvr;\O)$. 

On the other hand, let the flux $\bsigma = -A \nabla u+A\bff_2\in L^2(\O)^d$, then $\gradt\bsigma  =f_1 - \bb\cdot\nabla u-cu \in L^2(\O)$, thus we have $\bsigma \in H(\divvr;\O)$.
We use the fact $\nabla u = \bff_2 -A^{-1}\bsigma$ and get a first-order system:
\beq \label{fosys1_m1}
\left\{
\begin{array}{rclll}
\bsigma + A \nabla u &=&A\bff_2& \mbox{ in } \O,
\\[1mm]
\gradt\bsigma +Xu &=&f_1&  \mbox{ in } \O,\\[1mm]
u &=& 0 &\mbox{ on } \p\O.
\end{array}
\right.
\eeq
For $(\btau,v)\in H(\divvr;\O)\times H^1_0(\O)$, define the least-squares functional for the system \eqref{fosys1_m1},
\begin{eqnarray}
M(\btau,v;f_1,\bff_2) &:=& \|A^{-1/2}\btau + A^{1/2}\nabla v-A^{1/2}\bff_2\|_0^2 +\|\gradt\btau +Xv-f_1  |_0^2.
\end{eqnarray}
The corresponding least-squares minimization problem is:
$$
\mbox{Find } (\bsigma,u)\in H(\divvr;\O)\times H^1_0(\O) \mbox{ s.t. }
M(\bsigma,u;f_1,\bff_2) = \inf_{(\btau,v)\in H(\divvr;\O)\times H^1_0(\O)}M(\btau,v;f_1,\bff_2). 
$$
The Euler-Lagrange weak problem is: Find $(\bsigma,u)\in H(\divvr;\O)\times H^1_0(\O)$, such that
\beq\label{vp}
b((\bsigma,u), (\btau,v)) = F(\btau,v) \quad \forall(\btau,v)\in H(\divvr;\O)\times H^1_0(\O),
\eeq
where the bilinear form $a$ and linear form $F$ are defined for all  $(\brho,w)$ and $(\btau,v)\in H(\divvr;\O)\times H^1_0(\O)$ as:
\begin{eqnarray*}
b((\brho,w), (\btau,v)) &=& (\brho+A\nabla w, A^{-1}\btau+\nabla v)+ (\gradt\brho +Xw,\gradt\btau +Xv) ,\\
F(\btau,v) &=& (A\bff_2,  A^{-1}\btau+\nabla v)+ (f_1 ,\gradt\btau +Xv).
\end{eqnarray*}

Let $\cT = \{K\}$ be a triangulation of $\O$ using simplicial elements. The mesh $\cT$ is assumed to be regular. For an element $K\in \cT$ and an integer $k\geq 0$, let $P_k(K)$ be the space of polynomials with degrees less than or equal to $k$. Define the finite element spaces $RT_k$, $S_k$, and $S_{k,0}$ as follows:
$$
RT_k  :=\{\btau \in H(\divvr; \O) \colon
			\btau|_K \in P_k(K)^d + x P_k(K),\,\, \forall \, K\in \cT\},
$$
$$			
S_k :=	\{v \in C^0(\O) \colon
			v|_K \in P_k(K),\,\, \forall \, K\in \cT\}\quad\mbox{and}\quad S_{k,0} := S_k \cap H_0^1(\O).			
$$
Then for  $k\geq 0$ being an integer, the corresponding first-order system least-squares minimization problem and finite element problem are: Find $(\bsigma_h,u_h)\in  RT_k\times S_{k+1,0}$, such that
$$
M(\bsigma_h,u_h;f_1,\bff_2) = \inf_{(\btau,v)\in RT_k\times S_{k+1,0}}M(\btau,v;f_1,\bff_2), 
$$
and find $(\bsigma_h,u_h)\in  RT_k\times S_{k+1,0}$, such that
\beq\label{lsfem}
b((\bsigma_h,u_h), (\btau,v)) = F(\btau,v) \quad \forall(\btau,v)\in  RT_k\times S_{k+1,0},
\eeq
respectively.

Since it is obvious that
$
M(\btau,v;0,0) = L(\btau,v;0),
$
we immediately have 
\beq
b((\btau,v), (\btau,v)) \geq C\tri (\btau,v)\tri^2. 
\eeq
Thus the existence and uniqueness of both \eqref{vp} and \eqref{lsfem} are established. 

\subsection{A priori error estimate}
It is easy to derive that the quasi-best approximation property holds:
\beq\label{Cea}
\tri (\bsigma-\bsigma_h,u-u_h)\tri \leq C\inf_{(\btau_h,v_h)\in RT_k \times S_{k+1,0}} \tri (\bsigma-\btau_h,u-v_h)\tri.
\eeq
Thus, we only need the approximation properties of the discrete spaces and the regularities of the solutions to finish the a priori error estimate.

First, we discuss the local approximation properties of $S_{k+1,0}$.  The space $H^{1+s}(\O)$, with $s>0$ for two dimensions and $s > 1/2$ for three dimensions, is embedded in $C^0(\O)$ by Sobolev's embedding theorem. Hence, we can define the nodal interpolation of the function $v\in H^{1+s}(\O)$. It is proved in \cite{DuSc:80} that if $v \in H^{1+s_K}(K)$ with $s_K > 0$ in two dimensions and  $s_K > 1/2$ in three dimensions, then for $s_K\leq k+1$, the following estimate holds for the nodal interpolation $I_{no}$:
\beq\label{nodal}
\|\nabla(v-I_{no}v)\|_{0,K} \leq C h^{s_K}|v|_{1+s_K,K} \quad \forall K\in \cT.
\eeq
For solutions with low regularities, the nodal interpolation is not well-defined. We can use the modified Cl\'{e}ment interpolation \cite{Clement:75,BeGi:98} or the Scott-Zhang interpolation \cite{SZ:90}. For an element $K\in\cT$, let $\Delta_K$ be the collection of elements in $\cT$ that share at least one vertex with $K$.
Assume that $v\in H^1_0(\O)$ and $v|_{\Delta_K}\in H^{1+s_{\Delta_K}}(\Delta_K)$ for some $0 <  s_{\Delta_K}\leq k+1$, and let $I_{sz} v$ be the Scott-Zhang interpolation into $S_{k+1,0}$, we have 
\beq\label{sz}
\|\nabla(v-I_{sz} v)\|_{0,K} \leq C h^{s_{\Delta_K}}|v|_{1+s_{\Delta_K},\Delta_K}.
\eeq
Define $\cT_s$ to be the part of the mesh such that the local element-wise regularity $s_K$ of $H^{1+s_K}(K)$ is  big enough to ensure the nodal interpolation:
$$
\cT_s:=\{K\in \cT: s_K>0 \mbox{ for } d=2 \mbox{ and } s_K>1/2  \mbox{ for }  d=3 \}.
$$
Combining the approximation properties of \eqref{nodal} and \eqref{sz}, we have an almost localized approximation result: assume that $u\in H^{1}_0(\O)$,  $u|_K \in H^{1+s_K}(K)$ for $K\in \cT_s$,  and  $u|_{\Delta_K} \in H^{1+s_{\Delta_K}}({\Delta_K})$ for $K\in \cT\backslash\cT_s$, where $\max_{K\in\cT_s}\{s_K\}\leq k+1$ and $\max_{K\in\cT\backslash\cT_s}\{s_{\Delta_K}\}\leq k+1$,
\beq\label{uapp}
\inf_{v\in S_{k+1,0} } \|\nabla(u-v)\|_{0} \leq C (\sum_{K\in\cT_s}h_K^{s_K} |u|_{1+s_K,K}+\sum_{K\in\cT\backslash \cT_s}h_K^{s_{\Delta_K}} |u|_{1+s_{\Delta_K},\Delta_K}).
\eeq
Assume that $\bsigma = A(\bff_2-\nabla u)$ and $\gradt\bsigma = f_1 -\bb\cdot\nabla u-cu$ have the following local regularities, respectively:
$$
\bsigma|_K \in H^{\ell_K}(K) \quad\mbox{and}\quad \gradt\bsigma|_K \in H^{t_K}(K) \quad K\in \cT.
$$

For a fixed $r>0$, denote by $I_{rt}: \Hdiv \cap [H^r(\O)]^d \mapsto RT_k$ the standard $RT$ interpolation operator.  We have the following local approximation property: for $\btau \in H^{\ell_K}(K)$, $0<\ell_K \leq k+1$,
\begin{eqnarray} \label{rti}
\|\btau - I_{rt} \btau\|_{0,K}
  &\leq& C h_K^{\ell_K} |\btau|_{\ell_K,K} \quad\forall\,\, K\in \cT,
\end{eqnarray}
The estimate in (\ref{rti}) is standard for $\ell_K\geq 1$ and can be proved by the average Taylor series developed in \cite{DuSc:80} and the standard reference element technique with Piola transformation for $0<\ell_K<1$. We also should notice that the interpolations and approximation properties are entirely local.

Denote by $Q_k:  L^2 (\O) \mapsto D_k$ the $L^2$-projection onto $D_k:=\{v\in L^2(\O): v|_K \in P_k(K), K\in \cT\}$.  The following commutativity property is well-known:
$$
 \gradt (I_{rt}\,\btau)=Q_k\,\gradt\btau \qquad
 \forallqq\,\btau\in\Hdiv \cap H^r(\O)^d \,\mbox{ with }\, r>0.
$$
Thus we have the following local approximation property: for $\gradt\btau \in H^{t_K}(K)$, $0<t_K \leq k+1$,
\begin{eqnarray} \label{divrti}
\|\gradt(\btau - I_{rt} \btau)\|_{0,K}
  &\leq& C h_K^{t_K} |\gradt\btau|_{t_K,K} \quad\forall\,\, K\in \cT.
\end{eqnarray}

Combing the above approximation properties and \eqref{Cea}, we have the following a priori error estimate:
\begin{thm}{\bf (A priori error estimate)} Assume that $u\in H^{1}_0(\O)$,  $u|_K \in H^{1+s_K}(K)$ for $K\in \cT_s$,  and  $u|_{\Delta_K} \in H^{1+s_{\Delta_K}}({\Delta_K})$ for $K\in \cT\backslash\cT_s$, where $\max_{K\in\cT_s}\{s_K\}\leq k+1$ and $\max_{K\in\cT\backslash\cT_s}\{s_{\Delta_K}\}\leq k+1$. Assume that $\bsigma|_K \in H^{\ell_K}(K)$ and $\gradt\bsigma|_K \in H^{t_K}(K)$,  for $0<\ell_K \leq k+1$ and $0<t_K \leq k+1$. Then we have the following a priori error estimate:
\begin{eqnarray} \label{apriori_ls}
\tri (\bsigma-\bsigma_h,u-u_h)\tri &\leq &C (\sum_{K\in\cT_s}h_K^{s_K} |u|_{1+s_K,K}+\sum_{K\in\cT\backslash \cT_s}h_K^{s_{\Delta_K}} |u|_{1+s_{\Delta_K},\Delta_K} \\ \nonumber
&&+\sum_{K\in\cT}(h_K^{\ell_K} |\bsigma|_{\ell_K,K} + h_K^{t_K} |\gradt\bsigma|_{t_K,K})).
\end{eqnarray}
\end{thm}
\begin{rem}
We should notice that the smoothness $s_K$, $\ell_K$, and $t_K$ can be independent. Compared to the standard analysis, the localized a priori does not require global smoothnesses of $\bsigma$ and $\gradt\bsigma$, or in terms of the data, global smoothnesses of $A\bff_2$ and $f_1$. We only need the local (element-wise) smoothness of such data. 

The discussion of \eqref{uapp} can be applied to the conforming $C^0$-Lagrange approximation to the standard variational formulation too. However, the analysis is not coefficient-robust compared to the results in \cite{CHZ:17,Zhang:20mixed} for discontinuous and mixed approximations.

The a priori error estimate with local regularity is the base for adaptive finite element methods to achieve equal discretization error distribution. In a sense, the a priori results show us what an optimal mesh should be. For example, assuming $\cT_s = \cT$, then an error-equal-distributed mesh should have a similar size of 
$h_K^{s_K} |u|_{1+s_K,K} + h_K^{\ell_K} |\bsigma|_{\ell_K,K} + h_K^{t_K} |\gradt\bsigma|_{t_K,K})$ for all $K\in\cT$.

\end{rem}

\subsection{A posteriori error estimate} Let
\beq
E=\bsigma -\bsigma_h \quad \mbox{and}\quad e=u-u_h.
\eeq
Then by the first-order system \eqref{fosys1_m1}, we have
\begin{eqnarray*}
M(\bsigma_h,u_h;f_1,\bff_2) &=& \|A^{-1/2}(\bsigma_h+A\nabla u_h -\bsigma-A\nabla u)\|_0^2 
+ \|\gradt\bsigma +Xu - \gradt\bsigma_h +Xu_h \|_0^2 \\
&=& \|A^{-1/2}(E+A\nabla e)\|_0^2 + \|\gradt E +Xe\|_0^2 = M(E,e;0,\bzero).
\end{eqnarray*}
Since $E\in H_N(\divvr;\O)$ and $e\in H^1_0(\O)$, we have
\beq \label{rande}
C_1 \tri (E,e) \tri^2 \leq M(\bsigma_h,u_h;f_1,\bff_2) \leq C_2 \tri (E,e) \tri^2.
\eeq
Thus we can define a posteriori error estimator and local error indicator as:
\begin{eqnarray*}
\eta_K^2 &=& \|A^{-1/2}\btau + A^{1/2}\nabla v-A^{1/2}\bff_2\|_{0,K}^2 +\|\gradt\btau +Xv-f_1  |_{0,K}^2,\\
\eta^2 &=& \sum_{K\in\cT}\eta_K^2 = M(\bsigma_h,u_h;f_1,\bff_2).
\end{eqnarray*}
The equivalence \eqref{rande} shows the reliability and efficiency of the error estimator. 

By the triangle inequality,  we also have the local efficiency bound,
\begin{eqnarray*}
\eta_K &\leq & C(\|\nabla e\|_0+\|e\|_0 + \|E\|_0 + \|\gradt E\|_0) \\
 &\leq& 
C (h_K^{s_K} |u|_{1+s_K,K}+h_K^{s_{\Delta_K}} |u|_{1+s_{\Delta_K},\Delta_K}+h_K^{\ell_K} |\bsigma|_{\ell_K,K} + h_K^{t_K} |\gradt\bsigma|_{t_K,K}).
\end{eqnarray*}
Thus, even though we do not have the local error exactness of the indicator, if the error indicators $\eta_K$ are of a similar size,  then we can achieve the local optimal error estimate \eqref{apriori_ls}, which means the mesh obtained by an equal $\eta_K$ distribution for all $K\in\cT$ is optimal.

%

\bibliographystyle{siamplain}
\bibliography{../bib/szhang}

\end{document}